\documentclass[12pt]{amsart}

\usepackage{bm,color, ytableau, url}
\numberwithin{equation}{section}

\newtheorem{theorem}{Theorem}[section]
\newtheorem{proposition}[theorem]{Proposition}
\newtheorem{lemma}[theorem]{Lemma}
\newtheorem{corollary}[theorem]{Corollary}

\theoremstyle{definition}
\newtheorem{definition}[theorem]{Definition}

\newtheorem{example}[theorem]{Example}

\newtheorem{remark}[theorem]{Remark}

\begin{document}

\title{Gr\"obner fans of Specht ideals}

\author{Hidefumi Ohsugi}
\address{
Hidefumi Ohsugi,
Department of Mathematical Sciences, School of Science,
Kwansei Gakuin University, Sanda, Hyogo 669-1337, Japan
}
\email{ohsugi@kwansei.ac.jp}

\author{Kohji Yanagawa}
\address{
Kohji Yanagawa,
Department of Mathematics, Kansai University, Suita 564-8680, Japan.
}
\email{yanagawa@kansai-u.ac.jp}

\newcommand{\add}[1]{\ensuremath{\langle{#1}\rangle}}
\newcommand{\tab}[1]{\ensuremath{{\rm Tab}({#1})}}
\newcommand{\stab}[1]{\ensuremath{{\rm STab}({#1})}}
\newcommand{\ini}[1]{\ensuremath{{\rm in}({#1})}}
\newcommand{\Fc}{{\mathcal F}}
\newcommand{\Ib}{{\mathbf I}}
\newcommand{\Vb}{{\mathbf V}}
\newcommand{\wb}{{\mathbf w}}
\newcommand{\RR}{{\mathbb R}}
\newcommand{\hlam}{{\widehat{\lambda}}}

\def\doml{\vartriangleleft}
\def\domle{\trianglelefteq}

\newcommand{\fS}{{\mathfrak S}}
\newcommand{\sm}{{\mathsf m}}
\newcommand{\init}{\operatorname{in}}

\begin{abstract}
In this paper, we give the Gr\"obner fan and the state polytope of a Specht ideal $I_\lambda$ explicitly. In particular, we show that the state polytope of $I_\lambda$ 
for a partition $\lambda=(\lambda_1, \ldots, \lambda_m)$ is always a generalized permutohedron, and it is a (usual) permutohedron if and only if $\lambda_{i-1}=\lambda_i>0$ for some $i$.    
\end{abstract}

\maketitle

\section{Introduction}
Let $S=K[x_1,\dots,x_n]$ be a polynomial ring over a field $K$, $\lambda$  a partition of $n$, and $\tab \lambda$ the set of tableaux of shape $\lambda$.
For a tableau $T \in \tab \lambda$, we have its Specht polynomial $f_T \in S$ (see Definition~\ref{Specht polynomial} below). 
The $n$-th symmetric group $\fS_n$ naturally acts on the vector space spanned by $\{ f_T \mid T \in \tab \lambda \}$.
This $\fS_n$-module is called a {\bf  Specht module}, and plays a crucial role in the representation theory of symmetric groups, especially when $\operatorname{char}(K)=0$. In the present paper, we study the {\bf Specht ideal} $I_\lambda \subset S$, which is generated by  $\{ f_T \mid T \in \tab \lambda \}$. 
Specht ideals have been studied by several authors from several points of view (sometimes under other names).
See, for example, \cite{BGS,MW,MRV,SY2}.   

For an ideal $I \subset S$, a finite subset $G$ of $I$ is called a \textbf{universal Gr\"obner basis}
if $G$ is a Gr\"obner basis of $I$ with respect to any monomial order.
In their unpublished manuscript, Haiman and Woo found a   universal Gr\"obner basis of $I_\lambda$, 
and Murai and the authors of the present paper gave a short proof of this result (\cite{OMY}).  
See Theorem~\ref{HW} below. 
In the present paper, we study the number of all possible initial ideals of $I_\lambda$.  
The following is a main result of this paper. 

\medskip

\noindent{\bf Theorem~\ref{n!/(m+1)!}.}
{\it For a partition $\lambda=(\lambda_1, \ldots, \lambda_m)$ of $n$ with $\lambda_m>0$, set 
$k:=\min \{ \, \lambda_{i-1}-\lambda_i \mid i=2,3, \ldots, m \, \}$. }
Then $I_\lambda$ admits exactly $n!/(k+1)!$ distinct initial ideals under all possible monomial orders of $S$. 

\medskip

The Gr\"obner fan of an ideal is introduced by Mora and Robbiano \cite{MoRo}.
Although there is a good software package {\tt Gfan} \cite{GfanSoft} for computing Gr\"obner fans,
the computation is very difficult in general.
On the other hand, the state polytope of a homogeneous ideal is
introduced by Bayer and Morrison \cite{BaMo}.
It is a dual of a Gr\"obner fan and 
a generalization of the Newton polytope of a single homogeneous polynomial.

We give a brief introduction of these concepts. 
Detailed definitions will be introduced in Section 3.
Given a vector $\wb \in \RR^n$,
a partial order on the set of monomials in $S$ is
defined by $x_1^{a_1} \dots x_n^{a_n} > x_1^{b_1} \dots x_n^{b_n}$
if $(a_1,\dots, a_n) \cdot \wb > (b_1,\dots,b_n) \cdot \wb$.
For a homogeneous ideal $I \subset S$, let 
$\init_\wb (I)$ be the ideal generated by
the initial forms of polynomials in $I$ with respect to $\wb$.
Clearly, $\init_\wb (I)$ is not a monomial ideal in general. 
However, for any monomial order $<$, the set 
$
\{ \wb \in \RR^n \mid  \init_\wb (I) = \init_< (I) \}
$
is nonempty, and forms an open convex polyhedral cone.
The \textbf{Gr\"obner fan} ${\rm GF} (I)$ of $I$ is a polyhedral complex generated by the 
closures of these cones. 
A convex polytope $P \subset \RR^n$ is called a \textbf{state polytope} of $I$ 
if ${\rm GF} (I)$ is the normal fan of $P$. In particular, each initial ideal of $I$ corresponds to each vertex of the state polytope of $I$.
See, for example, \cite{CLO, Stu} for the details. 

\medskip

\noindent{\bf Theorem~\ref{genpermuto} and Corollary~\ref{permuto}.} 
{\it 
Let $\lambda=(\lambda_1, \ldots, \lambda_m)$ be a partition with $\lambda_m>0$. Then
the state polytope of $I_\lambda$ is a generalized permutohedron. 
In particular, 
 the state polytope is a (usual) permutohedron
 if and only if $\lambda_i=\lambda_{i-1}$ for some $i \le m$.
}

\section{Preliminaries}

\ytableausetup{centertableaux,boxsize=0.3em}
A {\bf partition} of a positive integer $n$ is a non-increasing sequence of non-negative  integers  $\lambda=(\lambda_1,\dots,\lambda_m)$ with $\lambda_1+ \cdots+\lambda_m=n$, but we identify $(\lambda_1,\dots,\lambda_m)$ with $(\lambda_1,\dots,\lambda_m,0)$. 
Therefore we frequently assume that $\lambda_m >0$. 
If $\lambda$ is a partition of $n$, then we write $\lambda \vdash n$. 

 A partition $\lambda \vdash n$ is represented by its Young diagram. 
 For example, $(4,2,1)$ is represented as $\ydiagram {4,2,1}$. 
 A (\textbf{Young}) \textbf{tableau} of shape $\lambda$ is a bijective filling of the Young diagram of  $\lambda$
by the integers in $\{1,2, \ldots, n\}$. 
For example,
$$
\ytableausetup{mathmode, boxsize=1.3em}
\begin{ytableau}
3 & 5 & 1& 7   \\
4 & 2    \\
6 \\
\end{ytableau}
$$
is a tableau of shape $(4,2,1)$. 
The box in the $i$-th row and the $j$-th column has the coordinates $(i,j)$, as in a matrix. For example, in the above tableau, the box in the $(3,1)$ position is filled by the number 6. 

\begin{definition}\label{Specht polynomial}
The {\bf Specht polynomial} $f_T$ of $T \in \tab \lambda$ is the product of all $x_i-x_j$ such that $i$ and $j$ are in the same column of $T$ and $j$ is in a lower position than $i$.
\end{definition}

For example, if $T$ is the above tableau, 
then $f_T=(x_3-x_4)(x_3-x_6)(x_4-x_6)(x_5-x_2).$ 
In this paper, we study the Specht ideal 
$$I_\lambda:=\langle \, f_T \mid T \in \tab \lambda \, \rangle \subset S$$
of $\lambda$.
If $\lambda=(n)$, then $I_\lambda$ is the trivial ideal $S$ itself. 
Therefore, in the rest of this paper, we assume that $\lambda_2 >0$. 

For partitions $\lambda=(\lambda_1,\dots,\lambda_m)$ and $\mu =(\mu_1,\dots,\mu_l)$ of $n$, we write $\mu \domle \lambda$ if $\lambda$ is larger than or equal to $\mu$ with respect to the dominance order,
that is,
$$ \mu_1+ \cdots + \mu_k  \leq \lambda_1+ \cdots+\lambda_k
\ \ \ \mbox{ for all }k. 
$$ By $\mu \doml \lambda$, we mean that $\mu \domle \lambda$ and 
$\mu \ne \lambda$. Now we can introduce an unpublished result of Haiman and Woo.  

\begin{theorem}[Haiman-Woo, c.f.~\cite{OMY}]\label{HW}
With the above situation,
$$\{ f_T \mid T \in \tab \mu, \mu \vdash n, \mu \domle \lambda \}$$
forms a universal Gr\"obner basis of $I_\lambda$. 
\end{theorem}


While the following lemma has been used in \cite{OMY}, the detailed discussion is given in (the proof of) \cite[Lemma~3.10]{RY}. 
The crucial point for this fact is that all elements of our Gr\" obner basis of $I_\lambda$ are products of linear forms.  
\begin{lemma}
\label{pickup}
Let $\init_<(I_\lambda)$ be the initial ideal of $I_\lambda$ with respect to a monomial order $<$ on $S$. 
Take the  permutation  $\sigma \in \fS_n$ with  $x_{\sigma(1)} < x_{\sigma(2)} < \cdots < x_{\sigma(n)}$. 
Then, for the lexicographic order $\prec$ with  $x_{\sigma(1)} \prec x_{\sigma(2)} \prec \cdots \prec x_{\sigma(n)}$, we have  $\init_< (I_\lambda)=\init_{\prec}(I_\lambda)$. 
\end{lemma}
 
 Unless otherwise specified, we fix $\sigma \in \fS_n$ and use the  lexicographic order $\prec$ with  $x_{\sigma(1)} \prec x_{\sigma(2)} \prec \cdots \prec x_{\sigma(n)}$.   
 Therefore we simply denote $\init_\prec(I_\lambda)$ by $\ini{I_\lambda}$.

If $i$ and $j$ are in the same column of $T$, 
we have $f_{\tau T} =-f_T$ for the transposition $\tau= (i \, j)$.  
In this sense, to consider $f_T$, we may assume that $T$ is \textbf{column standard} (with respect to $\sigma$), that is, all columns are increasing from top to bottom with respect to the order $\sigma(1) \prec \sigma(2) \prec \cdots \prec \sigma(n)$.  If a column standard tableau $T$ is also row standard (i.e., all rows are increasing from left to right with respect to $\prec$), we say $T$ is \textbf{standard}. Let $\stab \lambda$ be the set of all standard tableaux of shape $\lambda$.  It is a classical result that $\{ f_T \mid T \in \stab \lambda \}$ is a basis of the vector space spanned by $\{ f_T \mid T \in \tab \lambda \}$. 
Hence we have $I_\lambda= \langle \, f_T \mid T \in \stab \lambda \, \rangle.$ 
We also remark that, if $T$ is column standard and if the number $i$ is in the $d_i$-th row of $T$
for $i=1,2,\dots,n$, then we have $$\ini {f_T}=\prod_{i=1}^n x_i^{d_i-1}.$$
This equation is frequently used throughout the paper.

The next result follows from Theorem~\ref{HW} and \cite[Lemma~4.3.1]{L}. Note that the following set is still far from a minimal Gr\"obner basis in general. 

\begin{corollary}[{\cite[Remark~3.8]{OMY}}]\label{pre minimal GB}
With the above situation,
$$\left\{ f_T \mid T \in \stab \mu,\, \mu \vdash n, \, \mu \domle\lambda, \, \mu_1 =\lambda_1 \right\}$$
forms a Gr\"obner basis of $I_\lambda$ with respect to $\prec$. 
\end{corollary}

\begin{lemma}\label{gap for mu}
Let $\lambda=(\lambda_1, \ldots, \lambda_m)$ be a partition of $n$ with $\lambda_m>0$, and set $k:=\min \{ \, \lambda_{i-1}-\lambda_i \mid i=2,3, \ldots, m \, \}$. If $\ini{f_T}$ for $T \in \stab \mu$ with $\mu \domle \lambda$ is a minimal generator of $\ini{I_\lambda}$ and $\sigma(n)$ is in the $j$-th row of $T$, then either $j=1$ or $j\ge 2$ and $\mu_{j-1}-\mu_j \ge k$. 
\end{lemma}

\begin{proof} 
Assume that $j\ge 2$ and $(0\le ) \,\mu_{j-1}-\mu_j < k$. 
By Corollary~\ref{pre minimal GB}, we have $\mu \doml \lambda$ and $\mu_1 =\lambda_1$. 
Assume that $j=2$.
Since $\mu_2 \le \lambda_2$, it follows that $k > \mu_1-\mu_2 =\lambda_1-\mu_2 \ge \lambda_1-\lambda_2\ge k$, 
and this is a contradiction. Thus we have $j \ge 3$.
Since $\mu_1 = \lambda_1 > \lambda_2 \ge \mu_2$ and $j-1 \ge 2$,
there exists $1 < l < j$ such that 
$$
\mu_{l-1}>\mu_l =\mu_{l+1}=\cdots =\mu_{j-1}.
$$
Note that $\sum_{i=1}^{l-1} \mu_i \le \sum_{i=1}^{l-1} \lambda_i$. 
Next we will show that
\begin{equation}\label{L J-1}
    \sum_{i=1}^s \mu_i < \sum_{i=1}^s \lambda_i \mbox{ for all } l \le s \le j-1. 
\end{equation}

\noindent
First we prove it  for $l \le s \le j-2$ (when $l \le j-2$).
If $\sum_{i=1}^l \mu_i=\sum_{i=1}^l \lambda_i$,
then we have $\mu_l \ge \lambda_l$. Since $\lambda_l -  \lambda_{l+1} \ge k >0$ and $\mu_l=\mu_{l+1}$, we have $\mu_{l+1} > \lambda_{l+1}$. 
Hence $\sum_{i=1}^{l+1} \mu_i > \sum_{i=1}^{l+1} \lambda_i$,  but it contradicts $\mu \doml \lambda$. 
Thus we have $\sum_{i=1}^l \mu_i < \sum_{i=1}^l \lambda_i$. 
Similarly, for $l \le s \le j-2$, we have $\sum_{i=1}^s \mu_i < \sum_{i=1}^s \lambda_i$.  
It remains to show that $\sum_{i=1}^{j-1} \mu_i < \sum_{i=1}^{j-1} \lambda_i$. 
Assume the contrary, that is,  $\sum_{i=1}^{j-1} \mu_i = \sum_{i=1}^{j-1} \lambda_i$. 
Since $\sum_{i=1}^{j-2} \mu_i \le \sum_{i=1}^{j-2} \lambda_i$ (unless $l=j-1$, the inequality is strict), we have $\mu_{j-1} \ge \lambda_{j-1}$. Moreover, since
$$\mu_{j-1}-\mu_j < k \le \lambda_{j-1}-\lambda_j\le\mu_{j-1}-\lambda_j,$$
we have $\mu_j > \lambda_j$ and hence $\sum_{i=1}^j \mu_i > \sum_{i=1}^j \lambda_i$. This is a contradiction. 
Summing up, we have $\sum_{i=1}^s \mu_i < \sum_{i=1}^s \lambda_i$ for all $l \le s \le j-1$. 

Since $\sigma(n)$ is in the $j$-th row of $T \in \stab \mu$, we have $\mu_j >\mu_{j+1}$.  
We define the partition $\nu \vdash n$ by $\nu_l=\mu_l+1$, $\nu_j =\mu_j-1$, and $\nu_i=\mu_i$ for $i \ne l,j$ (since $\mu_{l-1}>\mu_l$ and $\mu_j >\mu_{j+1}$, $\nu$ is actually a partition).
By (\ref{L J-1}) above, we have $\mu \doml \nu \domle \lambda$. For the tableau $T$, lifting $\sigma(n)$ to the $l$-th row, we get a new tableau $T' \in \stab \nu$. Then $\ini{f_{T'}} \in \ini{I_\lambda}$ strictly divides $\ini{f_T}$, but it contradicts the assumption that $\ini{f_T}$ is a minimal generator of $\ini{I_\lambda}$. 
\end{proof} 
  
\begin{corollary}\label{cor of gap for mu}
With the same notation as Lemma~{\rm \ref{gap for mu}}, assume that $\sigma(n)$ is in (the right most box of) the $j$-th row
of $T \in \stab{\mu}$ with $j \ge 2$
and $\ini{f_T}$ is a minimal generator of $\ini{I_\lambda}$. 
If the number just above $\sigma(n)$ (i.e., the one in the $(j-1, \mu_j)$ position) is $\sigma(i)$, then we have $i < n-k$. 
\end{corollary}

\begin{proof}
Since $\mu_{j-1}-\mu_j \ge k$ by the lemma,
there are at least $k$ boxes in 
the right of the box filled by $\sigma(i)$.
These boxes are filled by $\sigma(l)$ for $i < l <n$.  
Since $\# \{ \sigma(l) \mid i < l <n\} =n-i-1 \ge k$,
we have $i \le n-k-1$.
\end{proof}
\begin{definition}
For a monomial $\sm :=\prod_{i=1}^n x_i^{a_i} \in S$, set $\deg_i \sm=a_i$.
For a partition $\lambda \vdash n$ and $1 \le i \le n$, set 
$$d_\lambda(i):=\sum_{T \in \stab \lambda} \deg_i (\ini {f_T}).$$
   \end{definition}

\begin{lemma}\label{d_lam}
With the above notation,
we have $$d_\lambda(\sigma(i))\le d_\lambda(\sigma(i+1))$$ for all $1 \le i <n$. Moreover, the inequality is strict if and only if there exists $T \in \stab \lambda$ such that $\sigma(i)$ and $\sigma(i+1)$ are in the same column of $T$. 
\end{lemma}

\begin{proof}
Take $T \in \stab \lambda$. 
If $\sigma(i)$ and $\sigma(i+1)$ are in the same row of $T$, we have $\deg_{\sigma(i)}( \ini{f_T})= \deg_{\sigma(i+1)} (\ini{f_T})$. 
Next, assume that  $\sigma(i)$ and $\sigma(i+1)$ are in different rows and different columns of $T$. For the transposition $\tau =(\sigma(i) \ \, \sigma(i+1))$, we have $\tau T \in \stab \lambda$, and 
$$\deg_{\sigma(i)} (\ini{f_T})+ \deg_{\sigma(i)} (\ini{f_{\tau T}})= \deg_{\sigma(i+1)} (\ini{f_T})+ \deg_{\sigma(i+1)} (\ini{f_{\tau T}}).$$ 
Finally, if $\sigma(i)$ and $\sigma(i+1)$ are in the same column of $T$, we have $\deg_{\sigma(i)} (\ini{f_T}) < \deg_{\sigma(i+1)} (\ini{f_T})$.
\end{proof}

For a partition $\lambda=(\lambda_1, \ldots,\lambda_m) \vdash n$ with $\lambda_1 \ge 2$, we define $\hlam=(\hlam_1, \ldots,\hlam_l) \vdash (n-1)$ inductively as follows. Set $\hlam_1 =\lambda_1-1$, and 
$$
\hlam_i = \min \left \{ \hlam_{i-1}, \ \sum_{j=1}^i \lambda_j -  \sum_{j=1}^{i-1} \hlam_j-1 \right \}
$$
for $i \ge 2$. 
For example, since $\lambda_1 + \lambda_2 -  \hlam_1-1 = \lambda_2$, we have
$\hlam_2= \min \{\lambda_1-1, \lambda_2\}$.
\begin{itemize}
    \item 
    If $\lambda_2 < \lambda_1$, then $\hlam_1=\lambda_1-1$ and $\hlam_i= \lambda_i$ for $i \ge 2$.
    We can prove it by induction.
   Assume that $i >2$ and $\hlam_k= \lambda_k$ for $2 \le k \le i-1$.
   Then
$\sum_{j=1}^i \lambda_j -  \sum_{j=1}^{i-1} \hlam_j-1= \lambda_i$.
Hence we have $\hlam_i= \min \{\lambda_{i-1}, \lambda_i\} = \lambda_i$.
    \item 
     It is easy to show that $\hlam_1=\hlam_2=\cdots =\hlam_i=\lambda_1-1$
      if $\lambda_1=\lambda_2=\cdots =\lambda_i$.  
\end{itemize} 
For $\nu \vdash (n-1) $, we define the partition $\overline{\nu}$ by 
$\overline{\nu}_1 = \nu_1+1$ and $\overline{\nu}_i = \nu_i $ for $i\ge 2$.

\begin{lemma}\label{additional_lemma}
For a partition $\lambda\vdash n$, set 
$$X:=\{ \, \nu \vdash (n-1) \mid \overline{\nu} \domle \lambda \, \}.$$
Then, for $\nu \vdash (n-1)$, $\nu \in X$ if and only if $\nu \domle \hlam$.  
\end{lemma}

\begin{proof}
Let $k:=\max \{i \mid \lambda_i=\lambda_1\}$. 
 Consider the partition $\rho \vdash (n-1)$ defined by $\rho_k=\lambda_k -1$ and $\rho_i=\lambda_i$ for $i \ne k$. 
  We also define $\tau=(\tau_1, \ldots, \tau_{s+1}) \vdash (n-1)$ by $\tau_i =\lambda_1 -1$ for $1 \le i \le s$ and $\tau_{s+1}=r$, where $s$ is the quotient and $r$ is the remainder when $n-1$ is divided by $\lambda_1 -1$.
Then it is easy to see that, for $\nu \vdash (n-1)$, 
\begin{equation}\label{useful_fact}
    \nu \in X \Longleftrightarrow \nu  \domle \rho, \tau. 
\end{equation}
In fact, for $\nu \vdash (n-1)$, $\nu_1 \le \lambda_1-1$ if and only if $\nu \domle \tau$. So we may assume that $\nu_1 \le \lambda_1-1$. 
We have $1+ \sum_{i=1}^j \nu_i =\sum_{i=1}^j \overline{\nu}_i$ for all $j$,
 $$
\sum_{i=1}^j \overline{\nu}_i \le j \lambda_1 = \sum_{i=1}^j \lambda_i=\sum_{i=1}^j \rho_i \quad \text{for } j <k,$$ 
and $ \sum_{i=1}^j \lambda_i=1+\sum_{i=1}^j \rho_i $ for all $j \ge k$. So $\overline{\nu} \domle \lambda$ if and only if $\nu \domle \rho$.

Since $\hlam_1 = \lambda_1-1$
and
$$ \sum_{j=1}^s \hlam_j 
\le   \sum_{j=1}^{s-1} \widehat{\lambda}_j  + \left(\sum_{j=1}^s \lambda_j -  \sum_{j=1}^{s-1} \hlam_j-1 \right) = \sum_{j=1}^s \lambda_j -1
\le \sum_{j=1}^s \rho_j$$
for $s \ge 2$, we have $\hlam \domle \rho, \tau$.
Thus $\hlam  \in X$ from (\ref{useful_fact}).

Since the set of partitions of $n-1$ forms a lattice with respect to the dominance order, $X$ has the maximum element $\rho \wedge \tau$. 
So it suffices to show that $\hlam = \rho \wedge \tau$. For this purpose, it suffices to show that $\hlam$ is a maximal element of $X$, equivalently, no element of $X$ covers $\hlam$.  By \cite[Proposition~2.3]{B}, if $\nu$ covers $\hlam$, then there are two integers $i, i'$ with $i< i'$ such that $\nu_i=\hlam_i+1$, $\nu_{i'}=\hlam_{i'}-1$, and $\nu_j=\hlam_j$ for all $j \ne i, i'$. 
If $i=1$ (resp. $i>1$), then $\nu_1 =\lambda_1$ (resp.  $\nu \not  \domle \rho$). 
From (\ref{useful_fact}), 
$\nu \not \in X$.
\end{proof}

\begin{lemma}\label{elimination} 
For a partition $\lambda\vdash n$,  
we have 
$$I_\hlam = I_\lambda \cap K[x_{\sigma(1)}, \ldots, x_{\sigma(n-1)}].$$
\end{lemma}

\begin{proof}
Set $S':= K[x_{\sigma(1)}, \ldots, x_{\sigma(n-1)}]$ and 
$J:=I_\lambda \cap S'$. 
Since our order $\prec$ is a lexicographic order, it is an  elimination order, that is, $\ini{f_T} \in S'$ implies $f_T \in S'$. By \cite[Proposition~15.29]{E}, 
\begin{equation*}
   {\mathcal G} = \{ \, f_T \mid T \in \stab \mu,\, \mu \vdash n, \, \mu \domle\lambda, \,  
    \ini{f_T} \in S' \, \}
\end{equation*}
is a Gr\"obner basis of $J$. 
Here, as the monomial order on $S'$, we use the restriction of $\prec$ to $S'$. Take $f_T$ for $T \in \stab \mu$ with $\mu \vdash n$ and  $\mu \domle\lambda$. 
Clearly, $\ini{f_T} \in S'$ if and only if $\sigma(n)$ is in the first row of $T$. If this is the case, we have $\mu_1 > \mu_2$. 
Moreover, removing $\sigma(n)$ from $T$, we get 
$T' \in \stab{\widehat{\mu}}$  
satisfying 
$f_T=f_{T'}$,
Thus
\begin{eqnarray*}
{\mathcal G} &=&  \{ \, f_{T'} \mid T' \in \stab {\widehat{\mu}},\, \mu \vdash n, \, \mu \domle\lambda, \, 
   \mu_1 > \mu_2 \, \}\\
     &=&  \{ \, f_{T'} \mid T' \in \stab {\nu},\, \nu \vdash (n-1), \, \overline{\nu} \domle\lambda\, \}.
\end{eqnarray*}
By Lemma \ref{additional_lemma}, we have 
$${\mathcal G} = \{ \, f_{T'} \mid T' \in \stab \nu,\, \nu \vdash (n-1), \, \nu \domle\hlam\, 
\}.$$
Since
this is a Gr\"obner basis of $I_\hlam$, 
we have $J=I_\hlam$. 
\end{proof}

\section{The proofs of the main results}
The following proposition is just a special case of Theorem~\ref{n!/(m+1)!} below. 
However, for better exposition, we prove it independently.   
Later, in Theorem~\ref{n!/(m+1)!} below, we will see that the converse of the proposition also holds.

\begin{proposition}\label{n!}
If a partition $\lambda=(\lambda_1, \ldots, \lambda_m)$ of $n$ satisfies  $\lambda_{i-1}=\lambda_i>0$ for some $i$, considering all monomial orders of $S$ 
{\rm (}equivalently, considering all $\sigma \in \fS_n${\rm )}, 
the Specht ideal $I_\lambda$ admits $n!$ distinct initial ideals.
\end{proposition}

\begin{proof}
We prove the statement by induction on $n$. 
By Lemma \ref{pickup},
it suffices to show 
that, if $\sigma, \sigma' \in \fS_n$ give the same initial ideal of $I_\lambda$, then we have $\sigma = \sigma'$,
%
in other words, we can recover a unique $\sigma  \in \fS_n$ from $\ini {I_\lambda}$.
Let $i= \max \{ j \mid \lambda_{j-1}=\lambda_j >0\}$.
Then there exists $T \in \stab \lambda$ such that $\sigma(n)$ (resp. $\sigma(n-1)$) is in the right most box of the $i$-th (resp. $(i-1)$-st) row, that is, in the $(i, \lambda_i)$ (resp. $(i-1, \lambda_i)$) position. 
Since $\sigma(n)$ and $\sigma(n-1)$ are in the same column of $T$, we have $d_\lambda(\sigma(n))> d_\lambda(\sigma(n-1))$, and hence $d_\lambda(\sigma(n))> d_\lambda(\sigma(i))$ for all $1 \le i <n$ by Lemma~\ref{d_lam}. 
Thus we can detect $\sigma(n)$, in other words, $\sigma(n) = \sigma'(n)$ if $\sigma, \sigma' \in \fS_n$ give the same initial ideal of $I_\lambda$.

As we have seen in the proof of   Lemma~\ref{elimination}, we have 
$$\ini{I_\hlam} = \ini{I_\lambda} \cap K[x_{\sigma(1)}, \ldots, x_{\sigma(n-1)}],$$
where we use the lexicographic order with $x_{\sigma(1)} \prec x_{\sigma(2)} \prec \cdots \prec x_{\sigma(n-1)}$  as the monomial order in $K[x_{\sigma(1)}, \ldots, x_{\sigma(n-1)}]$. 
If $\lambda_1>\lambda_2$, then we have $i\ge 3$ and $\hlam_{i-1}=\lambda_{i-1}=\lambda_i=\hlam_i$. 
If $\lambda_1=\lambda_2$, then $\hlam_1=\hlam_2$. 
Hence $\hlam$ always satisfies the assumption of the proposition. 
By induction hypothesis, we can detect each of $\sigma(1),\ldots, \sigma(n-1)$ from $\ini{I_\hlam}$ (hence, from $\ini{I_\lambda}$). 
\end{proof}

\begin{example}
If $\lambda=(2,2)$, $I_\lambda$ admits $4!$ distinct initial ideals, and we can recover the permutation $\sigma$ from $\ini {I_\lambda}$ by Proposition~\ref{n!}. However, it is easy to see that 
$$d_\lambda(\sigma(2))= d_\lambda(\sigma(3))= 1,$$
in other words, $d_\lambda(-)$ does not distinguish $\sigma(2)$ from $\sigma(3)$.  
Hence we have to consider $\mu$ with $\mu \doml \lambda$. 
\end{example}

We are in a position to give the number of distinct initial ideals of Specht ideals.

\begin{theorem}\label{n!/(m+1)!}
Let $\lambda=(\lambda_1, \ldots, \lambda_m)$ be a partition of $n$ with $\lambda_m>0$, and set 
$k:=\min \{ \, \lambda_{i-1}-\lambda_i \mid i=2,3, \ldots, m \, \}$. 
Then the Specht ideal $I_\lambda$ admits exactly $n!/(k+1)!$ distinct initial ideals under all monomial orders of $S$.
\end{theorem}

\begin{proof}
By Proposition~\ref{n!},
we may assume $k>0$, and hence $\lambda_1 > \cdots > \lambda_m$.
Take $p$ with $k=\lambda_{p-1} -\lambda_p$. 
To the Young diagram of $\lambda$, 
we put $\sigma(n-k)$ in the right most box of the $p$-th row (i.e., in the $(p, \lambda_p)$ position), 
and $\sigma(n-k-1)$ just above it (i.e., in the $(p-1, \lambda_p)$ position). 
In the right of the box filled by $\sigma(n-k-1)$, 
there are  $k$ boxes, and we fill them by  $\sigma(n-k+1), \sigma(n-k+2), \ldots, \sigma(n)$ 
in the suitable order. 
Filling the remaining boxes in a suitable way, we get 
$T \in \stab{\lambda}$ such that  $\sigma(n-k-1)$ and $\sigma(n-k)$ are in the same column of $T$.
Thus $d_\lambda(\sigma(n-k-1)) < d_\lambda(\sigma(n-k))$, and hence we can detect the sets
$A_1 := \{ \sigma(1), \sigma(2), \ldots, \sigma(n-k-1) \}$ 
and
$A_2:=\{ \sigma(n-k), \sigma(n-k+1), \ldots, \sigma(n) \}$
from $\ini{I_\lambda}$. 
Since $\sharp A_2 = k+1$,  it is enough to show that, from $\ini{I_\lambda}$,
\begin{itemize}
    \item[(a)]
we cannot distinguish $\sigma(n-k), \sigma(n-k+1), \ldots, \sigma(n)$ from each other,
and
    \item[(b)]
we can detect each of $\sigma(1), \sigma(2), \ldots, \sigma(n-k-1)$. 
\end{itemize}

\smallskip

\noindent
(a) 
It suffices to show the following statement. 
\begin{itemize}
\item [($*$)] If some two elements of $A_2$ are in the same column of  some  $T \in \stab{\mu}$ with $\mu \domle \lambda$, then $\ini{f_T}$ is not a minimal generator of $\ini{I_\lambda}$. 
\end{itemize}
In fact, if $(*)$ holds, we can find a generating set of $\ini{I_\lambda}$ which is stable under the action of $\fS_{A_2}$ by an argument similar to the proof of Lemma~\ref{d_lam}. 

The proof of $(*)$ is by induction on $n$. 
Assume that $\ini{f_T}$ with $T \in \stab{\mu}$ is a minimal generator of $\ini{I_\lambda}$ and some two elements $\sigma(\alpha), \sigma(\beta)
\in A_2$
are in the same column of $T$. 
By Corollary~\ref{cor of gap for mu}, 
we have $\sigma(n) \neq \sigma(\alpha), \sigma(\beta)$, and it implies that $k \ge 2$.
Assume that $\sigma(n)$ is in the $j$-th row of $T$. 
Removing $\sigma(n)$ from $T$, we have a standard tableau $T'$ of shape $\mu' \vdash (n-1)$ with $\mu'_j=\mu_j -1$ and $\mu'_i=\mu_i$ for $i \ne j$. 
Since $\lambda_j >\lambda_{j+1}$, 
we can define the partition $\lambda' \vdash (n-1)$ by $\lambda'_j=\lambda_j-1$ and $\lambda'_i=\lambda_i$ for $i \ne j$. 
Clearly, $\mu' \domle \lambda'$.

Since $k':=\min \{ \, \lambda'_{i-1}-\lambda'_i \mid i=2,3, \ldots, m \, \} \ge k-1$ and 
$$
\sigma(\alpha), \sigma(\beta)
\in
A_2 \setminus \{ \sigma(n)\}=\{\sigma(n-k), \ldots, \sigma(n-1)\}
\subset \{\sigma((n-1)-k'), \ldots, \sigma(n-1)\}
,$$ we can apply the induction hypothesis, that is, the condition $(*)$ for $T'$ and $\ini{I_{\lambda'}}$. 
Hence $\ini{f_{T'}}$ is not a minimal generator of $\ini{I_{\lambda'}}$. 
In other words, there exists $\nu' \vdash (n-1)$ with $\nu' \domle \lambda'$, and $T_1' \in \stab{\nu'}$ such that $\ini{f_{T'_1}}$ strictly divides $\ini{f_{T'}}$. 
Since $\sum_{i=1}^j v'_i \le \sum_{i=1}^j \lambda'_i < \sum_{i=1}^j \lambda_i$, by an argument similar to the proof of Lemma~\ref{gap for mu}, we can find  $l \le j$ such that the sequence $\nu=(\nu_1, \nu_2, \ldots)$ given by $\nu_l=\nu'_l+1$ and $\nu_i=\nu'_i$ for $i \ne l$ is a partition of $n$ satisfying $\nu \domle \lambda$. Adding $\sigma(n)$ to the $l$-th row of $T_1'$, we get a standard tableau $T_1 \in \stab{\nu}$.   
Note that  $\ini{f_T} = \ini{f_{T'}} \cdot x_{\sigma(n)}^{j-1}$ and 
$\ini{f_{T_1}} = \ini{f_{T_1'}} \cdot x_{\sigma(n)}^{l-1} $. 
Since $l \le j$  and   $\ini{f_{T'_1}} $ strictly divides $\ini{f_{T'}}$, we have $\ini{f_{T_1}} \, ( \in \ini{I_\lambda})$ strictly divides $\ini{f_T}$. This contradicts the assumption that $\ini{f_T}$ is a minimal generator of $\ini{I_\lambda}$.

\smallskip

\noindent
(b) 
Set 
$
l:= \lambda_1 - \lambda_2$ ($ \ge k$). 

\smallskip

\noindent
\textbf{Case 1}: Assume that $l >k$.
First, we will show that
\begin{equation}\label{n-l to n-k}
d_\lambda(\sigma(n-l)) < d_\lambda(\sigma(n-l+1)) < \cdots < d_\lambda(\sigma(n-k-1)).
\end{equation}
It suffices to show that, for $j$ with $n-l < j < n-k$, there exists $T \in \stab{\lambda}$ such that  $\sigma(j-1)$ and $\sigma(j)$ are in the same column of $T$.   
Take $p$ with $k=\lambda_{p-1} -\lambda_p$. 
Then $p\ge 3$.
To the Young diagram of $\lambda$, we put $\sigma(j)$ in the right most box of the $p$-th row (i.e., in the $(p, \lambda_p)$ position), and $\sigma(j-1)$ just above it (i.e., in the $(p-1, \lambda_p)$ position). In the right of the box filled by $\sigma(j-1)$, there are  $k\, (=\lambda_{p-1} -\lambda_p)$ boxes, and we fill them by  $\sigma(j+1), \sigma(j+2), \ldots, \sigma(j+k)$ in the  suitable order. Next, we fill the boxes in the first row from the right most one by $\sigma(j+k+1), \sigma(j+k+2), \ldots, \sigma(n)$ in the suitable order. Since 
$$\# \{\sigma(j+k+1), \sigma(j+k+2), \ldots, \sigma(n) \} =n-j-k < l,$$
these numbers are contained in the ``peninsula" part of the first row. 
Filling the remaining boxes in a suitable way, we get a desired standard tableau.

By \eqref{n-l to n-k}, we can detect the set $\{ \sigma(1), \sigma(2), \ldots, \sigma(n-l) \}$ from $\ini{I_\lambda}$. 
Hence we can take the ideal 
$I_\lambda \cap K[x_{\sigma(1)}, x_{\sigma(2)}, \ldots, x_{\sigma(n-l)}]$, which equals $I_\mu$ by the repeated use of Lemma~\ref{elimination}. 
Here $\mu$ is the partition of $(n-l)$ given by $\mu_1=\lambda_1-l=\lambda_2$ and $\mu_i=\lambda_i$ for $i \ge 2$. 
Since $\mu_1=\mu_2$, we can detect each of $\sigma(1), \ldots, \sigma(n-l)$ from $\ini{I_\mu}$ (hence, from $\ini{I_\lambda}$) by Proposition~\ref{n!}. 
Combining with $\eqref{n-l to n-k}$, we can detect each of  $\sigma(1), \sigma(2), \ldots, \sigma(n-k-1)$ from $\ini{I_\lambda}$.  

\smallskip

\noindent
\textbf{Case 2}: Assume that $l= k$.
Recall that we can detect the set
$A_1 = \{ \sigma(1), \sigma(2), \ldots$, $\sigma(n-k-1) \}$.
As we have seen in (a),
we cannot detect $\sigma(n-k)$ from $A_2$,
in other words, the variables $x_{\sigma(n-k)}, \dots, x_{\sigma(n)}$ appear 
in the initial ideal $\ini{I_\lambda}$ in the same way.
Take any $r \in A_2$, and consider 
$$I_\mu =I_\lambda \cap K[x_{\sigma(1)}, x_{\sigma(2)}, \ldots, x_{\sigma(n-k-1)}, x_r],$$
where $\mu \vdash (n-l)$ is the partition given in Case 1. 
By Proposition~\ref{n!}, we can detect each of $\sigma(1), \sigma(2), \ldots, \sigma(n-k-1)$ from $\ini{I_\lambda}$. 
\end{proof}

Let $\{0\} \ne I \subset K[x_1,\ldots,x_n]$ be a homogeneous ideal.
Given a vector $\wb \in \RR^n$,
let $\init_\wb (I)$ denote the initial form ideal of $I$ with respect to $\wb$.
Note that $\init_\wb (I)$ is not necessarily a monomial ideal.
For example, if $\wb = {\bf 0}$, then  $\init_\wb (I)=I$.
Since $I$ is homogeneous, for any vector $\wb_1 \in \RR^n$, 
there exists a nonnegative vector $\wb_2 \in \RR^n$ such that 
$\init_{\wb_1} (I) = \init_{\wb_2} (I)$.
Given a vector $\wb \in \RR^n$, let
$$
C[\wb]:= \{ \wb' \in \RR^n \mid  \init_\wb (I) = \init_{\wb'} (I) \}.
$$
In general, $C[\wb]$ is a relatively open convex polyhedral cone (\cite[Proposition~2.3]{Stu}).
A \textbf{fan} is a polyhedral complex consisting of the cones from the origin.
Let
$$
{\rm GF}(I) := 
\left\{\left.
\overline{C[\wb]} \ \right| \  \wb \in \RR^n
\right\},
$$
where $\overline{C[\wb]}$ is the closure of $C[\wb]$.
Then ${\rm GF}(I)$ is a fan, and called the \textbf{Gr\"obner fan} of $I$.
Note that ${\rm GF}(I)$ is \textbf{complete}, i.e., 
$$\bigcup_{C \in {\rm GF}(I)} C  = \RR^n.$$
The \textbf{normal fan} of a polytope $P \subset \RR^n$ is a fan
that is dual to $P$.
A convex polytope $P \subset \RR^n$ is called a \textbf{state polytope} of $I$ 
if ${\rm GF} (I)$ is the normal fan of $P$.
There is a one to one correspondence between the initial ideals of $I$ and the vertices of the state polytope of $I$.

Given a vector $(u_1,\dots,u_n) \in \RR^n$ 
let $P_n (u_1,\dots,u_n)$ be the convex hull of the set
$$
\{
(u_{\sigma(1)}, u_{\sigma(2)}, \dots, u_{\sigma(n)})
\in \RR^n \mid \sigma \in \fS_n
\}
.
$$
In particular, $\Pi_n :=P_n(1,2,\dots,n)$ is called  the \textbf{permutohedron}
of order $n$.
It is known that
the normal fan of the permutohedron $\Pi_n$ is
the \textbf{braid fan} ${\rm Br}_n$ that is the complete fan in $\RR^n$ given by the
hyperplanes $x_i - x_j = 0$ for all $i\ne j $.
Each maximal cone of ${\rm Br}_n$ is of the form
$$
\{ \wb \in \RR^n \mid w_{\sigma(1)} \le w_{\sigma(2)}  \le \cdots \le 
w_{\sigma(n)} 
\}
$$
for some $\sigma \in \fS_n $.
See \cite[Section 3.2]{PRW} for details.
For $0 \le k < n$, let 
$$\Pi_{n,k} := P_n(1,2,\ldots,n-k-1, n-k,\dots,n-k).$$
Then each maximal cone of the normal fan of $\Pi_{n,k}$ is of the form 
$$
C_{\sigma,k} :=
\{ \wb \in \RR^n \mid w_{\sigma(1)} \le w_{\sigma(2)} \le \cdots \le w_{\sigma(n-k)},\dots,w_{\sigma(n)}
\}
$$
for some $\sigma \in \fS_n$.
A \textbf{generalized permutohedron} \cite{Pos}
is a polytope
obtained by moving the vertices of
a permutohedron while keeping the same edge directions. 

\begin{proposition}[{\cite{PRW}}]
\label{gp}
A polytope $P\subset \RR^n$ is 
a generalized permutohedron if and only if the normal fan of $P$ is
refined by the braid fan ${\rm Br}_n$.
\end{proposition}

Since ${\rm Br}_n$ refines the normal fan of $\Pi_{n,k}$ for all $0 \le k <n$,
by Proposition~\ref{gp},
each $\Pi_{n,k}$ is a generalized permutohedron.

\begin{theorem}\label{genpermuto}
Let $\lambda=(\lambda_1, \ldots, \lambda_m)$ be a partition of $n$ with $\lambda_m>0$, and set 
$k=\min \{ \, \lambda_{i-1}-\lambda_i \mid i=2,3, \ldots, m \, \}$. 
Then the generalized permutohedron $\Pi_{n,k}$ is a state polytope of $I_\lambda$.
In particular, ${\rm GF}(I_\lambda)$ is the normal fan of $\Pi_{n,k}$, and
refined by the braid fan ${\rm Br}_n$.
\end{theorem}

\begin{proof}
From Proof of Theorem~\ref{n!/(m+1)!}, 
$I_\lambda$ admits exactly $n!/(k+1)!$ distinct initial ideals, and
monomial orders for each initial ideal correspond to the cone
$$
\{ \wb \in \RR^n \mid w_{\sigma(1)} < w_{\sigma(2)}  < \cdots < w_{\sigma(n-k)} , \dots ,w_{\sigma(n)} 
\}
$$
for some $\sigma \in \fS_n$.
Since its closure is $C_{\sigma,k}$, it follows that 
 ${\rm GF}(I_\lambda)$ is the normal fan of $\Pi_{n,k}$
as desired.
\end{proof}

\begin{remark}
In the situation of Theorem~\ref{genpermuto}, the largest  possible value of $k$ is $n-2$, which occurs when $\lambda=(n-1,1)$ (we are assuming that $\lambda \ne (n)$). 
In this case, the state polytope $\Pi_{n,n-2}$ is an $(n-1)$-simplex. We also remark that the dimension of $\Pi_{n,k}$ is always $n-1$ for all $0 \le k \le n-2$. 
\end{remark}

\begin{corollary}\label{permuto}
Let $\lambda=(\lambda_1, \ldots, \lambda_m)$ be a partition of $n$ with $\lambda_m>0$.
If $\lambda_{i-1}=\lambda_i>0$ for some $i$, then 
the permutohedron $\Pi_n$ of order $n$ is a state polytope of $I_\lambda$,
and ${\rm GF}(I_\lambda)$ is the braid fan ${\rm Br}_n$.
\end{corollary}

We close this paper with a computational experiment obtained by the software
{\tt Gfan} \cite{GfanSoft}.
Let $\lambda = (3,2,1)$ be a partition of $6$.
We prepare the input file {\tt input321.txt} whose contents are started with
$$Q[x1,x2,x3,x4,x5,x6]
\{x1^2 x2 x4 - x1 x2^2 x4 - x1^2 x3 x4 +  x2^2 x3 x4 + x1 x3^2 x4 - \cdots$$
and input

\medskip

\noindent
{\tt
gfan\_bases <input321.txt >output321.txt
}\\
{\tt
gfan\_bases <input321.txt | gfan\_leadingterms -m >outputinitial.txt
}

\medskip

\noindent
to {\tt Gfan}. 
Then the output shows that there are $360=6!/2!$ distinct initial ideals of $I_\lambda$,
and each initial ideal is generated by $17$ monomials. 
Among these 17 elements, 16 of them have degree 4, and one of them has degree 6. The element of degree 6 corresponds to the standard tableau 
$$
\ytableausetup{mathmode, boxsize=2em}
\begin{ytableau}
\sigma(1) & \sigma(5) & \sigma(6)  \\
\sigma(2) \\
\sigma(3) \\
\sigma(4) \\
\end{ytableau}.
$$
The files {\tt input321.txt}, {\tt output321.txt} and {\tt outputinitial.txt}
are available at 

\begin{center}
\url{https://drive.google.com/drive/folders/} \hspace{5cm} \ \\
\hspace{2cm} \url{1yQF0zXZIeyUkNqTBfoO_seyOlVIw3oxC?usp=drive_link}
\end{center}

\bigskip

\noindent
\textbf{Acknowledgments}:
The authors are grateful to professors Yasuhide Numata, Satoshi Murai, Ryota Okazaki, Kosuke Shibata, Akihito Wachi and Junzo Watanabe
for the useful discussions 
at the MFO-RIMS Tandem Workshop ``Symmetries on polynomial ideals and varieties".
In particular, Lemma~\ref{elimination} was given in the discussion. 
The first author is partially supported by KAKENHI 18H01134. 
The second author is partially supported by KAKENHI 22K03258.

\end{document}